\documentclass{article}

\usepackage{amsmath,amsthm,amssymb,amsfonts}
\usepackage{paralist}
\usepackage{stmaryrd}
 \usepackage{float}
\usepackage{rotating}
\usepackage{epstopdf}
\usepackage{color}

\usepackage{pgfplots}
\usepackage{pgfplotstable}

\newtheorem{theorem}{Theorem}[section]
\newtheorem{remark}[theorem]{Remark}
\newtheorem{lemma}[theorem]{Lemma}

\newtheorem{assumption}[theorem]{Assumption}

\newtheorem{myalg}{Algorithm}

 \newcommand{\sign}{\text{sign}}

\def \R{\mathbb{R}}

\def \N{\mathbb{N}}
\def \eps{\varepsilon}
\def \Uad{{U_{\text{ad}}}}

\numberwithin{equation}{section}

\begin{document}

\title{A priori stopping rule for an iterative Bregman method for optimal control problems\footnote{This work was founded by German Research Foundation DFG under project grant Wa 3626/1-1.}}

\author{Frank P\"orner\footnote{Department of Mathematics,  University of W\"urzburg, Emil-Fischer-Str. 40, 97074 W\"urzburg, Germany, E-mail: frank.poerner@mathematik.uni-wuerzburg.de}
}
\date{}
\maketitle

\begin{abstract}
In this article we continue our investigation of the iterative regularization method for optimization problems based on Bregman distances. The optimization problems are subject to pointwise inequality constraints in $L^2(\Omega)$. We provide an estimate for the noise error for perturbed data, which can be used to construct an a priori stopping rule. Furthermore we show how to implement our method with a semi-smooth Newton method using finite elements and present numerical results for the stopping rule.


\end{abstract}

\section{Introduction}
We consider an optimization problem of the following form:
\begin{equation}\label{eq:main_problem}\tag{$\textbf{P}$}
 \begin{split}
    \text{Minimize} &\quad \frac{1}{2}\|Su - z\|_Y^2  \\
    \text{such that} &\quad u_a \leq u \leq u_b \quad \text{a.e. in } \Omega,
\end{split}
\end{equation}
which can be interpreted both as an optimal control problem or as an inverse problem. Here $\Omega \subseteq \R^n$, $n \geq 1$ is a bounded, measurable set, $Y$ a Hilbert space and $z \in Y$ a given function. The operator $S: L^2(\Omega) \to Y$ is supposed to be linear and continuous. The inequality constraints are prescribed on the set $\Omega$. We assume $u_a, u_b \in L^\infty(\Omega)$. A common example of such an operator $S$ is the solution operator of a linear partial differential equation. In many situations the operator $S$ is compact or has non-closed range, which makes \eqref{eq:main_problem} ill-posed.\\
We want to see \eqref{eq:main_problem} as an optimal control problem. The unknown $u$ is the control and the constraints are limitations arising from the underlying physical problem. The given function $z$ is the desired state, and we search for a control $u$ satisfying the constraints, such that $Su$ is as close to $z$ as possible. In many situations $z$ cannot be reached due to the presence of the control constraints.\\
Though solutions of \eqref{eq:main_problem} exist, uniqueness of solutions can only be proven under additional assumptions, e.g. injectivity of $S$. Furthermore solutions may be unstable with respect to perturbations, which is critical if only approximations $z^\delta \approx z$ of the exact data $z$ exist.\\
In order to overcome these difficulties, several regularization methods were developed. The most common is the Tikhonov regularization with some positive regularization parameter $\alpha > 0$. The regularized problem is given by:
\begin{equation*}
 \begin{split}
    \text{Minimize} &\quad \frac{1}{2}\|Su - z^\delta\|_Y^2 + \frac{\alpha}{2}\|u\|^2 \\
    \text{such that} &\quad u_a \leq u \leq u_b \quad \text{a.e. in } \Omega,
\end{split}
\end{equation*}
where $z^\delta$ with $\|z-z^\delta\| \leq \delta$ is the perturbed state to the noise level $\delta \geq 0$. Here one is interested in the convergence of the solution for $(\alpha, \delta) \to 0$ under some suitable conditions. For this problem convergence results were developed in \cite{wachsmuth2011}. In the context of inverse problems, we refer to \cite{engl1996}. However, for $\alpha$ tending to zero, the Tikhonov regularized problem becomes increasingly ill-conditioned.\\

In order to overcome this difficulty, we started in \cite{wachsmuth2016} to investigate the Bregman iterative regularization technique. There, the iterate $u_{k+1}$ is given by the solution of
\begin{equation*}
\text{Minimize} \quad \frac{1}{2}\|Su-z\|_Y^2 + \alpha_{k+1}  D^{\lambda_{k}}(u,u_{k}),
\end{equation*}
where $D^\lambda(u,v) := J(u) - J(v) - (u-v,\lambda)$ is called the (generalized) Bregman distance associated with a regularization function $J$ with subgradient $\lambda \in \partial J(v)$.
This iteration method was first used first in \cite{burger2007,osher2005}, it was
applied to an image restoration problem with $J$ being the total variation.
We choose to incorporate the control constraint into the regularization functional, resulting in
\[
J(u) := \frac{1}{2}\|u\|^2 + I_\Uad(u),
\]
where $\Uad = \{u \in L^2(\Omega): \; u_a \leq u \leq u_b\}$, and $I$ is the indicator function of convex analysis.
While at first sight the incorporation of $I_\Uad$ into the Bregman regularization functional together with the explicit control constraint $u\in U_{ad}$ seems to be redundant, this choice allows to prove strong convergence under a suitable regularity assumption, which allows bang-bang structure and non-attainability (see \cite{wachsmuth2016}). We recall the most important results in section \ref{sec:bregman}, including our regularity assumption.

The convergence and regularization error estimates are formulated assuming that the value $z$ is known exactly. If only approximations $z^\delta \approx z$ are available, the next iterate $u_{k+1}^\delta$ is given by the solution of
\begin{equation*}
\text{Minimize} \quad \frac{1}{2}\|Su-z^\delta\|_Y^2 + \alpha_{k+1}  D^{\lambda_{k}^\delta}(u,u_{k}^\delta).
\end{equation*}
Again we assume that the noise level $\delta$ is known and $z^\delta$ satisfies $\|z-z^\delta\| \leq \delta$. In general we cannot expect convergence of the sequence $(u_k^\delta)_k$. Our aim is to identify an optimal parameter $k(\delta)$, at which it is reasonable to stop the iteration.\\

In section \ref{sec:noise_error} we derive an estimate for the noise error $\|u_k-u_k^\delta\|$ which is used to construct an a priori stopping rule. Furthermore convergence of $u_{k(\delta)}^\delta$ is investigated as $\delta \to 0$.  In Section \ref{sec:numerics} we show how to use and how to implement a semi-smooth Newton solver into our iterative method using finite elements. Finally numerical results will be presented in section \ref{sec:numerical_results}.

\paragraph*{Notation.}
For elements $q \in L^2(\Omega)$, we denote the $L^2$-Norm by $\|q\| := \|q\|_{L^2(\Omega)}$. Furthermore $c$ is a generic constant, which may change from line to line, but is independent from the important variables, e.g. $k$.

\section{Bregman iteration}
\label{sec:bregman}

In order to prove convergence and convergence rates of our numerical method we need to assume some regularity of the solution. A common assumption on a solution $u^\dagger$ is the following source condition, which is an abstract smoothness condition, see, e.g.,  \cite{burger2007,chaventkunisch94,itojin11,neubauer1988,wachsmuth2011,wachsmuth2011b}. We say $u^\dagger$ satisfies the source condition \ref{ass:SC} if the following assumption holds.

{
\renewcommand{\thetheorem}{\textbf{SC}}
\begin{assumption}[Source Condition]\label{ass:SC}
Let $u^\dagger$ be a solution of \eqref{eq:main_problem}.
Assume that there exists an element $w \in Y$ such that $u^\dagger = P_\Uad(S^\ast w)$ holds.
\end{assumption}
}

The source condition is equivalent to the existence of Lagrange multipliers for the problem
\begin{equation}\label{eq:mini_norm}\begin{split}
\min\limits_{u \in \Uad} \quad & \frac{1}{2} \|u\|^2\\
\text{such that} \quad & Su = y^\dagger,
\end{split}
\end{equation}
where $y^\dagger$ is the uniquely defined optimal state of \eqref{eq:main_problem}. Note that the existence of Lagrange multipliers is not guaranteed in general, as in may situations the operator $S$ is compact or has non-closed range.\\
If $z$ is not attainable, i.e. $y^\dagger \neq z$, a solution $u^\dagger$ may be bang-bang, i.e., $u^\dagger$ is a linear combination of characteristic functions, hence discontinuous in general with $u^\dagger \not \in H^1(\Omega)$. But in many examples the range of $S^\ast$ contains $H^1(\Omega)$ or $C(\bar \Omega)$. Hence the source condition \ref{ass:SC} is too restrictive for this case. We will resort to the following condition. We say $u^\dagger$ satisfies the source condition \ref{ass:ActiveSet} if the following assumption holds. Recall that the adjoint state is defined by $p^\dagger = S^\ast(z-Su^\dagger)$.

{
\renewcommand{\thetheorem}{\textbf{ASC}}
\begin{assumption}[Active Set Condition]\label{ass:ActiveSet}
Let $u^\dagger$ be a solution of \eqref{eq:main_problem} and assume that there exists a set $I \subseteq \Omega$, a function $w \in Y$, and positive constants $\kappa, c$ such that the following holds

\begin{enumerate}
  \item (source condition) $I \supset \{ x \in \Omega: \; p^\dagger(x) = 0 \}$  and
  $$\chi_I u^\dagger = \chi_I P_\Uad (S^\ast w),$$
  \item (structure of active set) $A := \Omega \setminus I$ and for all $\eps > 0$
  $$|\{ x\in A: \; 0 < |p^\dagger(x)| < \eps  \}| \leq c \eps^\kappa,$$
  \item (regularity of solution) $S^\ast w \in L^\infty(\Omega)$.
\end{enumerate}
\end{assumption}
}

Assumption \ref{ass:ActiveSet} is a generalization of assumption \ref{ass:SC}, since for $I = \Omega$ both assumptions coincide. We decided to differentiate between them, because assumption \ref{ass:SC} omits more regularity, allowing us to establish improved results. This source condition is used in e.g. \cite{wachsmuth2011,wachsmuth2011b,wachsmuth2013,wachsmuth2016}.

In \cite{wachsmuth2016} we applied the Bregman iteration with the regularization functional
$$J: L^2(\Omega) \to  \R \cup \{-\infty, +\infty\}, \quad J(u) := \frac{1}{2}\|u\|^2 + I_\Uad(u)$$
where $I_C$ denotes the indicator function of the set $C$. The Bregman distance for $J$ at $u,v \in L^2(\Omega)$ and $\lambda \in \partial J(v)$ is defined as
$$D^\lambda(u,v) := J(u) - J(v) - (u-v, \lambda).$$
Here $\partial J(v)$ denotes the subdifferential of $J$ at $v$. The functional $J$ is convex and nonnegative, the Bregman distance is also nonnegative and convex with respect to $u$. Our method is now given by:

\begin{myalg}\label{alg:MinEx}
Let $u_0 = P_\Uad(0)  \in \Uad$, $\mu_0 = 0$, $\lambda_0 = 0  \in \partial J(u_0)$ and $k=1$.
\begin{enumerate}
  \item Solve for $u_k$: \label{a_start}
  \begin{equation}
\label{Min1xx}\tag{Min}
\text{Minimize} \quad \frac{1}{2}\|Su-z\|_Y^2 + \alpha_{k}  D^{\lambda_{k-1}}(u,u_{k-1}).
\end{equation}
  \item Set $\mu_k := \sum\limits_{i=1}^k \frac{1}{\alpha_i} (z-Su_i)$ and $\lambda_k := S^\ast \mu_k$.
  \item Set $k:=k+1$, go back to \ref{a_start}.
\end{enumerate}
\end{myalg}

Here $(\alpha_k)_k$ is a non-negative, uniformly bounded sequence of real numbers. Algorithm \ref{alg:MinEx} is well-posed, see \cite{wachsmuth2016}. We define the abbreviation
$$\gamma_k := \sum\limits_{j=1}^k \frac{1}{\alpha_j}.$$
The following theorem provides some regularization error estimates for the control $u_k$ under some suitable regularity assumptions. For the proof and for some general convergence results of algorithm \ref{alg:MinEx} we refer to \cite{wachsmuth2016}.

\begin{theorem}\label{thm:SC_strong_conv}
Let $(u_k)_k$ be the sequence generated by algorithm \ref{alg:MinEx}. Assume that Assumption \ref{ass:SC} holds for $u^\dagger$. Then
\begin{align*}
\|u^\dagger - u_k\|^2 &= \mathcal{O}(\gamma_k^{-1}) \quad \text{and} \quad \sum\limits_{i=1}^k \frac{1}{\alpha_i} \|u^\dagger - u_i\|^2 \leq c.
\end{align*}
If we assume that instead Assumption \ref{ass:ActiveSet} holds, then
\begin{align*}
\|u^\dagger - u_k\|^2 &= \mathcal{O}\left( \gamma_k^{-1} + \gamma_k^{-1} \sum\limits_{j=1}^k \alpha_j^{-1} \gamma_j^{- \kappa}   \right)\\
\text{and} \quad \sum\limits_{i=1}^k \frac{1}{\alpha_i} \|u^\dagger - u_i\|^2 &\leq c \left( 1 + \sum\limits_{i=1}^k \alpha_i^{-1} \gamma_i^{-\kappa} \right).
\end{align*}
\end{theorem}

Note that by the uniform boundedness of the sequence $(\alpha_k)_k$ and by \cite[Lemma 3.5]{wachsmuth2016} we obtain
$$\lim\limits_{k \to \infty} \gamma_k^{-1} = 0 \quad \text{and} \quad \lim\limits_{k\to \infty} \gamma_k^{-1} \sum\limits_{j=1}^k \alpha_j^{-1} \gamma_j^{- \kappa} = 0.$$

\section{Noise error estimate and stopping rule}
\label{sec:noise_error}

Assume that we do not know the exact data $z$, but rather a disturbed approximation $z^\delta$, which satisfies
$$\|z-z^\delta\|_Y \leq \delta.$$
The number $\delta \geq 0$ can be considered as an estimate for the noise level. Let $u^\dagger$ be a solution of \eqref{eq:main_problem} with the exact data $z$. We cannot expect $u_k^\delta \to u^\dagger$ if $\delta > 0$, even if some regularity assumption holds for $u^\dagger$. Here $(u_k^\delta)_k$ denotes the sequence generated by algorithm \ref{alg:MinEx} for $z^\delta$.

As pointed out in \cite{frick2010, frick2011, frick2012}
the Bregman iteration algorithm \ref{alg:MinEx} can be interpreted as an augmented Lagrange method
applied to the minimum norm problem:
$$\min\limits_{u \in \Uad} \|u\| \quad \text{subject to } Su=z.$$
Furthermore we want to point out, that the authors in \cite{frick2010} derived a stopping rule $k_M(\delta)$ based on Morozov's principle. One major assumption to enforce convergence of the iterates read (see \cite[Theorem 5.3]{frick2010}):
$$\lim\limits_{\delta \to 0} \delta^2 \sum\limits_{j=1}^{ k_M(\delta)} \frac{1}{\alpha_j} = 0 \quad \text{and} \quad \lim\limits_{\delta \to 0}\sum\limits_{j=1}^{k_M(\delta)} \frac{1}{\alpha_j} = \infty.$$
Then each weak cluster point of the sequence $(u_{k_M(\delta)}^\delta)_{\delta \to 0}$ is a solution of the original problem. The proof relies heavily on the attainability of $z$, and the source condition \eqref{ass:SC}. We cannot use this result due to non-attainability and the more general regularity assumption \eqref{ass:ActiveSet}.

\subsection{Noise estimate}
We establish the following noise estimate, which will be used later to construct the stopping rule.

\begin{lemma}\label{lem:noise_error}
Let $(u_k)_k$ and $(u_k^\delta)_k$ denote the sequences generated by Algorithm \ref{alg:MinEx} for data $z$ and $z^\delta$, respectively. Then it holds
$$\sum\limits_{i=1}^k \frac{1}{\alpha_i} \|u_i^\delta - u_i\|^2 \leq \delta^2 \sum\limits_{i=1}^k \left( \frac{1}{\alpha_i^2} + \gamma_{i-1}^2 \right).$$
\end{lemma}

\begin{proof}
We start by using the first order optimality conditions, both for $u_k$ and $u_k^\delta$ (compare to \cite{wachsmuth2016})
\begin{align*}
(-p_{k+1}^\delta + \alpha_{k+1}(u_{k+1}^\delta - \lambda_k^\delta) , u_{k+1} - u_{k+1}^\delta) &\geq 0,\\
(-p_{k+1} + \alpha_{k+1}(u_{k+1} - \lambda_k) , u_{k+1}^\delta - u_{k+1}) &\geq 0.
\end{align*}
By adding we obtain
$$\alpha_{k+1}\|u_{k+1} - u_{k+1}^\delta\|^2 \leq (p_{k+1} - p_{k+1}^\delta, u_{k+1} - u_{k+1}^\delta) + \alpha_{k+1}(\lambda_k - \lambda_k^\delta, u_{k+1} - u_{k+1}^\delta).$$
An estimate yields for the first term
\begin{align*}
(p_{k+1} - p_{k+1}^\delta, u_{k+1} &- u_{k+1}^\delta) = (z - Su_{k+1} - (z^\delta - Su_{k+1}^\delta), S(u_{k+1} - u_{k+1}^\delta))\\
&= (z-z^\delta, y_{k+1} - y_{k+1}^\delta) + (y_{k+1}^\delta - y_{k+1}, y_{k+1} - y_{k+1}^\delta)\\
&\leq \delta \|y_{k+1} - y_{k+1}^\delta\|_Y - \|y_{k+1} - y_{k+1}^\delta\|_Y^2,
\end{align*}
while for the second term we estimate
\begin{align*}
(\lambda_k - \lambda_k^\delta, u_{k+1} &- u_{k+1}^\delta) = \sum\limits_{i=1}^k \frac{1}{\alpha_i} (z-z^\delta - Su_i + Su_i^\delta, y_{k+1} - y_{k+1}^\delta)\\
&= \sum\limits_{i=1}^k \frac{1}{\alpha_i}(z-z^\delta, y_{k+1} - y_{k+1}^\delta) + \sum\limits_{i=1}^k \frac{1}{\alpha_i}(y_i^\delta - y_i, y_{k+1} - y_{k+1}^\delta)\\
&\leq \delta \gamma_k \|y_{k+1} - y_{k+1}^\delta\|_Y + \left( \sum\limits_{i=1}^k \frac{1}{\alpha_i} (y_i^\delta - y_i), y_{k+1} - y_{k+1}^\delta   \right).
\end{align*}
By defining the quantity
$$v_k := \sum\limits_{i=1}^k \frac{1}{\alpha_i}(y_i - y_i^\delta),$$
and using the equality
$$(-v_k, v_{k+1} - v_k) = \frac{1}{2}\|v_k\|_Y^2 - \frac{1}{2}\|v_{k+1}\|_Y^2 + \frac{1}{2}\|v_{k+1} - v_k\|_Y^2,$$
we obtain
\begin{align*}
(\lambda_k &- \lambda_k^\delta, u_{k+1} - u_{k+1}^\delta) \leq \delta \gamma_k \|y_{k+1} - y_{k+1}^\delta\|_Y + \alpha_{k+1} (-v_k, v_{k+1} - v_k)\\
&=  \delta \gamma_k \|y_{k+1} - y_{k+1}^\delta\|_Y + \alpha_{k+1} \left( \frac{1}{2}\|v_k\|_Y^2 - \frac{1}{2}\|v_{k+1}\|_Y^2 + \frac{1}{2}\|v_{k+1} - v_k\|_Y^2  \right).
\end{align*}
Putting everything together yields
\begin{align*}
\alpha_{k+1} \|u_{k+1} - u_{k+1}^\delta\|^2 &\leq  \delta \|y_{k+1} - y_{k+1}^\delta\|_Y^2 - \|y_{k+1} - y_{k+1}^\delta\|_Y^2 \\
&\quad+\alpha_{k+1} \delta \gamma_k \|y_{k+1} - y_{k+1}^\delta\|_Y \\
&\quad + \alpha_{k+1}^2 \left( \frac{1}{2}\|v_k\|_Y^2 - \frac{1}{2}\|v_{k+1}\|_Y^2 + \frac{1}{2}\|v_{k+1} - v_k\|_Y^2 \right).
\end{align*}
With
$$\|v_{k+1} - v_k\|_Y^2 = \frac{1}{\alpha_{k+1}^2}\|y_{k+1} - y_{k+1}^\delta\|_Y^2,$$
we obtain
\begin{align*}
\alpha_{k+1} \|u_{k+1} - u_{k+1}^\delta\|^2 &\leq \delta^2 + \frac{1}{4}\|y_{k+1} - y_{k+1}^\delta\|_Y^2 - \|y_{k+1} - y_{k+1}^\delta\|_Y^2\\
&\quad + \alpha_{k+1}^2 \delta^2 \gamma_k^2 + \frac{1}{4}\|y_{k+1} - y_{k+1}^\delta\|_Y^2\\
&\quad + \alpha_{k+1}^2 \left( \frac{1}{2}\|v_k\|_Y^2 -  \frac{1}{2}\|v_{k+1}\|_Y^2\right) + \frac{1}{2}\|y_{k+1} - y_{k+1}^\delta\|_Y^2\\
&= \delta^2 + \alpha_{k+1}^2 \delta^2 \gamma_k^2 +\alpha_{k+1}^2 \left( \frac{1}{2}\|v_k\|_Y^2 -  \frac{1}{2}\|v_{k+1}\|_Y^2\right).
\end{align*}
By dividing everything by $\alpha_{k+1}^2$ and performing a summation over $k$ yield the result
$$\sum\limits_{i=1}^k \frac{1}{\alpha_i} \|u_i^\delta - u_i\|^2 \leq \delta^2 \sum\limits_{i=1}^k \left( \frac{1}{\alpha_i^2} + \gamma_{i-1}^2 \right).$$
\end{proof}

\begin{remark}
The first iteration step is precisely a Tikhonov regularization with regularization parameter $\alpha_1$, so we should recover the same noise estimates. This is the case, since for $k=1$ we obtain
$$\|u_1^\delta - u_1\| \leq \frac{\delta}{\sqrt{\alpha_1}},$$
which is the same estimate obtained for Tikhonov with regularization parameter $\alpha_1$, see \cite[Theorem 3.1]{wachsmuth2011}.
\end{remark}

\begin{remark}
A slight modification of the proof above yields
$$\frac{1}{4}\sum\limits_{i=1}^k \frac{1}{\alpha_i^2} \|y_i - y_i^\delta\|_Y^2 + \sum\limits_{i=1}^k \frac{1}{\alpha_i} \|u_i^\delta - u_i\|^2 \leq 2 \delta^2 \sum\limits_{i=1}^k \left( \frac{1}{\alpha_i^2} + \gamma_{i-1}^2 \right),$$
from which we recover the estimates
\begin{align*}
\|u_1 - u_1^\delta\| &\leq c \frac{\delta}{\sqrt{\alpha_1}},\\
\|y_1 - y_1^\delta\| &\leq c \delta,
\end{align*}
which resembles the estimates obtained for the Tikhonov regularization but with a constant $c \geq 1$, see also \cite[Theorem 3.1]{wachsmuth2011}.
\end{remark}

\subsection{A priori stopping rule}

We will now combine the error estimates with respect to the noise level and regularization. This will give an a priori stopping rule with best possible convergence order. We assume that assumption \ref{ass:ActiveSet} holds for $u^\dagger$. The two estimates are given by (see lemma \ref{lem:noise_error} for the noise error and theorem \ref{thm:SC_strong_conv} for the regularization error):
\begin{align*}
\sum\limits_{i=1}^k \frac{1}{\alpha_i} \|u_i - u_i^\delta\|^2 &\leq \delta^2 \sum\limits_{i=1}^k \left( \frac{1}{\alpha_i^2} + \gamma_{i-1}^2 \right) =: e_k^n,\\
\sum\limits_{i=1}^k \frac{1}{\alpha_i}\|u^\dagger - u_i\|^2 &\leq c \left( 1 +   \sum\limits_{i=1}^k \alpha_i^{-1}\gamma_i^{-\kappa} \right) =: c e_k^r,
\end{align*}
where the sum of quadratic noise error $e_k^n$ and the sum of quadratic regularization error $e_k^r$ is defined by:
\begin{align*}
e_k^n &:= \delta^2 \sum\limits_{i=1}^k \left( \frac{1}{\alpha_i^2} + \gamma_{i-1}^2 \right),\\
e_k^r &:=  1 + \sum\limits_{i=1}^k  \alpha_i^{-1} \gamma_i^{-\kappa}.
\end{align*}
Our stopping rule is now given by: Find maximal $k(\delta)$, such that the noise error $e_i^n$ is below the regularization error $e_i^r$ for all $i \leq k(\delta)$. Hence the optimal parameter $k(\delta)$ is defined by
$$k(\delta) := \begin{cases}
0 & \text{if }e_1^n > \tau e_1^r\\
\max \{ k \in \N: \; e_{i}^n \leq \tau e_{i}^r, \quad \forall i \leq k \} & \text{else.}
\end{cases}$$
Here $\tau > 0$ is a constant. For the case $e_1^n > \tau e_1^r$ we define $k(\delta) = 0$, which reflects the case that the noise error is dominating after the first iteration. This happens only if $\delta$ is too big and we will show that $k(\delta) \neq 0$ for $\delta$ small enough. Note that $k(\delta)$ depends also on $\tau$ and $(\alpha_k)_k$, but we are suppressing the dependence due to clarity of the notation.

\begin{lemma}
Let $\delta > 0$. The value $k(\delta)$ defined above is well-defined and unique. Furthermore $k(\delta) \to \infty$ as $\delta \to 0$.
\end{lemma}

\begin{proof}
For the case $e_1^n > \tau e_1^r$, there is nothing to show. Now assume that $e_1^n \leq \tau e_1^r$ holds. We now show that there exists a $\bar k \in \N$ such that $e_{\bar k}^n > \tau e_{\bar k}^r$. Assume that such a value does not exists, hence we get $e_k^n \leq \tau e_k^r$ for all $k \in \N$. Multiplying this inequality with $\gamma_k^{-1}$ yields for $k \to \infty$ (see \cite[Lemma 3.5]{wachsmuth2016})
$$ \delta^2 \gamma_{k}^{-1} \sum\limits_{i=1}^{k} \left( \frac{1}{\alpha_i^2} + \gamma_{i-1}  \right) \leq \tau  \gamma_{k}^{-1} \left(1+   \sum\limits_{i=1}^{k} \left( \alpha_i^{-1}\gamma_i^{-\kappa}  \right) \right) \to 0.$$
Hence the sequence $(\gamma_k^{-1} e_k^n)_k$ tends to zero. Recall that there exists a constant $C > 0$ such that $\alpha_j \leq C$. Define the following quantities
$$\beta_j := C^{-1} \alpha_j, \quad \bar \gamma_k := \sum\limits_{i=1}^k \frac{1}{\beta_i},$$
leading to $0 < \beta_j \leq 1$. Now compute
\begin{align*}
\gamma_k^{-1} \sum\limits_{i=1}^k \frac{1}{\alpha_i^2} &= C^{-1} \left[ \left( C \gamma_k \right)^{-1}  \left( C^2 \sum\limits_{i=1}^k \frac{1}{\alpha_i^2}  \right) \right]\\
&= C^{-1} \bar \gamma_k^{-1} \sum\limits_{i=1}^k \frac{1}{\beta_i^2}\\
& \geq C^{-1} \bar \gamma_k^{-1} \sum\limits_{i=1}^k \frac{1}{\beta_i}\\
&= C^{-1}.
\end{align*}
We now have a contradiction since
\begin{align*}
0 < \delta^2 C^{-1} &\leq \limsup\limits_{k \to \infty} \delta^2 \gamma_{k}^{-1} \sum\limits_{i=1}^{k} \frac{1}{\alpha_i^2} \leq \limsup\limits_{k \to \infty} \delta^2 \gamma_{k}^{-1} \sum\limits_{i=1}^{k} \left( \frac{1}{\alpha_i^2} + \gamma_{i-1}^2  \right)\\
&\leq \tau \lim\limits_{k \to \infty} \gamma_{k}^{-1} \left(1+   \sum\limits_{i=1}^{k}( \alpha_i^{-1}\gamma_i^{-\kappa}  ) \right)\\
&= 0.
\end{align*}
Therefore, we know the existence of $\bar k$ with $e_{\bar k}^n > \tau e_{\bar k}^r$, and we can deduce the existence of a maximal $k^\ast < \bar k$ with $e_{i}^n \leq \tau e_{i}^r \; \forall i \leq k^\ast$. Setting $k(\delta) := k^\ast$ yields the well-posedness of $k(\delta)$.\\
To show the second part we assume that this is wrong, hence there exists a $\bar k \in \N$ such that for all $\delta > 0$ we have $k(\delta) < \bar k$. Without loss of generality we assume $k(\delta) \geq 1$ and $k(\delta) +1= \bar k $. By definition of $k(\delta)$ we now obtain
$$e_{\bar k}^n  =   \delta^2 \sum\limits_{i=1}^{\bar k}\left( \frac{1}{\alpha_i^2} + \gamma_{i-1}^2  \right) > \tau \left( 1 + \sum\limits_{i=1}^{\bar k} \alpha_i^{-1} \gamma_i^{-\kappa} \right) = \tau e_{\bar k}^r, \quad \forall \delta > 0.$$
This gives a contradiction for $\delta$ small enough.
\end{proof}

If we chose $k(\delta)$ based on the principle above, we can establish the following convergence result for $u_{k(\delta)}^\delta$ as $\delta \to 0$.
\begin{theorem}\label{thm:asymptotic_convergence}
Let $k(\delta)$ be given by the a priori stopping rule presented above. Then
$$\min\limits_{j=1,...,k(\delta)}\|u^\dagger - u_j^\delta\| \to 0$$
as $\delta \to 0$.
\end{theorem}

\begin{proof}
We use triangle inequality to obtain
\begin{align*}
\sum\limits_{i=1}^{k(\delta)} \frac{1}{\alpha_i} \|u^\dagger - u_{i}^\delta\|^2 &\leq \sum\limits_{i=1}^{k(\delta)} \frac{1}{\alpha_i} \left( \|u^\dagger - u_{i}\| + \|u_{i} - u_{i}^\delta\|  \right)^2\\
&\leq c \left( \sum\limits_{i=1}^{k(\delta)} \frac{1}{\alpha_i} \|u^\dagger - u_{i}\|^2 + \sum\limits_{i=1}^{k(\delta)} \frac{1}{\alpha_i}  \|u_i - u_i^\delta\|^2  \right)\\
&\leq c\left( e_{k(\delta)}^r + e_{k(\delta)}^n\right)\\
&\leq c(1+\tau )e_{k(\delta)}^r,
\end{align*}
which yields
$$\min\limits_{i=1,...,k(\delta)} \|u^\dagger - u_i^\delta\|^2 \leq c(1+\tau) \gamma_{k(\delta)}^{-1} e_{k(\delta)}^r$$
and since $k(\delta) \to \infty$ as $\delta \to 0$ we obtain the result (see \cite[Lemma 3.5]{wachsmuth2016}).
\end{proof}

The results can be improved if we assume that Assumption \ref{ass:SC} is satisfied. In this case we set $e_k^r := 1$, see theorem \ref{thm:SC_strong_conv}. Note that all of the results above stay true in this case. The modification of the proofs is simple.

\section{Numerical implementation}
\label{sec:numerics}
This section is devoted to the numerical implementation of the Bregman iteration using finite elements.

\subsection{Semi-Smooth Newton Method for the subproblem}
In our algorithm we need to solve the subproblem

\begin{equation}\label{eq:sub}\tag{SUB}
\left. \begin{split}
    \text{Minimize} &\quad  \frac{1}{2}\|Su-z\|_Y^2 + \alpha \left[ \frac{1}{2}\|u\|^2 -(\lambda,u) \right]  \\
    \text{such that} &\quad u \in \Uad
\end{split} \quad \right\}
\end{equation}
with $\lambda \in L^2(\Omega)$, which has a unique solution, characterized by the projection formula
\begin{equation}\label{eq:fixed}
u = P_\Uad \left( -\frac{1}{\alpha} p(u)+ \lambda \right),
\end{equation}
with $p(u) = S^\ast(Su-z) $. Several different techniques are available to solve \eqref{eq:sub}. The simplest is a projected gradient (see \cite{troelsch2010}) method with the decent direction $-S^\ast(Su -z) - \alpha (u - \lambda)$. The implementation is rather simple but comes at very slow convergence speed and high numerical costs. Nevertheless the gradient method can be used to globalize the Newton method presented below. 

In order to solve \eqref{eq:sub} we want to apply a Newton method to \eqref{eq:fixed}. In this section we follow the idea presented in \cite{hinze2009b}, where a semi-smooth Newton solver was applied for a Neumann-type elliptic optimal control problem. We adapt this technique for distributed control problems. This technique can also be applied for optimal boundary control problems, see \cite{beuchler2012}. Denote by $u^k$ the iterates given by the Newton method. Define the function
$$F(u) := u - P_\Uad \left(- \frac{1}{\alpha}p(u) + \lambda \right)$$
and apply a Newton step
$$0 = F(u^k) + G(u^k)(u^{k+1} - u^k),$$
where $G(u^k): L^2(\Omega) \to L^2(\Omega)$ is the Newton derivative of $F$ at $u^k$. For a convergence analysis of this Newton method we refer to \cite{hinze2009b}. A suitable function $G$ is given by the following lemma. The result can also be found in \cite{beuchler2012} or in \cite[Theorem 2.14]{hinze2009}.

\begin{lemma}\label{lemma:Newton_derivation}
A suitable function $G$ is given by
$$G(u^k)(u^{k+1} - u^k) = (u^{k+1} - u^k) + d \cdot \left( \frac{1}{\alpha} \left( p(u^{k+1}) - p(u^k)  \right) \right)$$
and
$$d = \begin{cases} 0, &\text{if } -\frac{1}{\alpha} p(u^k) + \lambda \geq u_b\\
1, &\text{if } -\frac{1}{\alpha} p(u^k) + \lambda \in (u_a, u_b)\\
0, &\text{if } -\frac{1}{\alpha} p(u^k) + \lambda \leq u_a \end{cases}. $$
\end{lemma}

We see that $u^{k+1}$ satisfies the relation
$$u^{k+1} = \begin{cases} u_b, &\text{if } -\frac{1}{\alpha} p(u^k) + \lambda \geq u_b \\
u_a, &\text{if }  -\frac{1}{\alpha} p(u^k) + \lambda \leq u_a\\
-\frac{1}{\alpha} p(u^{k+1}) + \lambda, &\text{if }  -\frac{1}{\alpha} p(u^k) + \lambda \in (u_a, u_b) \end{cases}.$$

Define the sets
\begin{align*}
A_b(u) &= \left\{ x \in \Omega: \; -\frac{1}{\alpha} p(u) + \lambda \geq u_b \right\},\\
I(u) &= \left\{x \in \Omega: \; -\frac{1}{\alpha} p(u) + \lambda \in (u_a, u_b) \right\},\\
A_a(u) &= \left\{x \in \Omega : \; -\frac{1}{\alpha} p(u) + \lambda \leq u_a \right\}
\end{align*}
and the operator
$$E_M : L^2(\Omega) \to L^2(\Omega), \quad E_M(v) = \chi_M (v).$$
We have $u^{k+1} = u_a$ on $A_a(u^k)$ and $u^{k+1} = u_b$ on $A_b(u^k)$. On the set $I(u^k)$ we obtain
$$E_{I(u^k)} \left( u^{k+1}  + \frac{1}{\alpha} p(u^{k+1}) - \lambda \right) = 0.$$
This can be rewritten in a linear equation for $E_{I(u^k)} u^{k+1}$.

\begin{lemma}\label{lemma:Control_lin_sys}
The function $u_{I}^{k+1} := E_{I(u^k)} u^{k+1}$ satisfies
\begin{equation}\label{equ:Control_lin_sys}
u_{I}^{k+1} + \frac{1}{\alpha} E_I q(u_{I}^{k+1}) = - \frac{1}{\alpha} E_{I(u^k)} p \big( E_{A_a(u^k)} u_a + E_{A_b(u^k)}u_b \big) + E_{I(u^k)} \lambda
\end{equation}
with $q(u) := S^\ast S u$.
\end{lemma}

Our Newton solver now solves the equation above for $u_{I}^{k+1}$, which allows us to construct our new iterate $u^{k+1}$.

\subsubsection{Algorithmic aspects and implementation}\label{sec:implementation}
We now focus on the special case where $y=Su$ is given by the elliptic equation
\begin{equation*}
\left. \begin{split}
    -\Delta y = u  &\quad \text{in } \Omega\\
    y = 0 &\quad \text{on } \partial \Omega
\end{split} \quad \right\}.
\end{equation*}
The discretized version is now given by the solution $(u_h,y_h,p_h)$ of the coupled problem
\begin{equation}\label{eq:num_coupled}
\begin{alignedat}{2}
a(y_h,v_h) &= ( u_h,v_h ), &\forall v_h \in \mathbb{V}_h\\
a(p_h,v_h) &= ( v_h,y_h-z_h ), &\forall v_h \in \mathbb{V}_h\\
u_h &= P_\Uad \left( -\frac{1}{\alpha} p_h + \lambda_h \right),
\end{alignedat}
\end{equation}
with the test function space $\mathbb{V}_h$, the bilinear form
$$a(w,v) := \int\limits_\Omega \nabla w \cdot \nabla v \; dx,$$
 and $\lambda_h \in V_h$. For a given $u_h$ there exists a unique $y_h(u_h)$ and hence a unique $p_h(u_h)$, so we reduce the coupled system \eqref{eq:num_coupled} to one equation for the optimal control $u_h$, e.g.
$$u_h = P_\Uad \left( -\frac{1}{\alpha} p_h(u_h) + \lambda_h \right).$$
Note that lemma \ref{lemma:Control_lin_sys} also holds for $u_h$. We are interested in the solution $u_{h,I}^{k+1}$ from equation \eqref{equ:Control_lin_sys}. But $u_{h,I}^{k+1}$ is not a finite element function in general, since it is the truncation of a finite element function $u_{h,I}^{k+1} = E_{I(u_h^k)} \tilde u_h^{k+1}$, which can be computed by solving
\begin{align}\label{eq:discr_uk1newton}
\begin{split}
E_{I(u_h^k)} \tilde u_h^{k+1} &+ \frac{1}{\alpha} E_{I(u_h^k)} q \left(E_{I(u_h^k)} \tilde u_h^{k+1} \right)\\
&= - \frac{1}{\alpha} E_{I(u_h^k)} p \big( E_{A_a(u_h^k)} u_a + E_{A_b(u_h^k)}u_b \big)+ E_{I(u_h^k)} \lambda_h.
\end{split}
\end{align}
In the following we denote by $\underline{u_h} \in \R^m$ the coefficient vector of a function $u_h \in V_h$, where $m$ denotes the degrees of freedom ($DOF$). By testing \eqref{eq:discr_uk1newton} with a test function we obtain the following lemma.

\begin{lemma}
The coefficient vector $\underline{\tilde u_h^{k+1}}$ satisfies
\begin{equation}\label{equ:Control_Matrix_Sys}
\big(M_I + \frac{1}{\alpha} M_I K^{-1} M K^{-1} M_I \big) \underline{\tilde u_h^{k+1}} = - \frac{1}{\alpha} M_I K^{-1}M \big( K^{-1}g - \underline{z} \big) + M_I \underline{\lambda_h},
\end{equation}
where
\begin{align*}
K &= \left[ \int_{\Omega} \nabla \phi_i \cdot \nabla \phi_j \right]_{ij},\\
M_I &= \left[ \int_{I(u_h^k)} \phi_i \phi_j \right]_{ij}, \quad M_{A_a} = \left[ \int_{A_a(u_h^k)} \phi_i \phi_j \right]_{ij},\\
M &= \left[ \int_\Omega \phi_i \phi_j \right]_{ij}, \quad M_{A_b} = \left[ \int_{A_b(u_h^k)} \phi_i \phi_j \right]_{ij},\\
g &= \left[ \int_{A_a(u_h^k)} u_a \phi_j + \int_{A_b(u_h^k)} u_b \phi_j \right]_j = M_{A_a} \underline{u_a} + M_{A_b} \underline{u_b}.
\end{align*}

\end{lemma}

Note that we now have the relation
$$u_h^{k+1} = E_{A_a(u_h^k)}u_a + E_{A_b(u_h^k)}u_b + E_{I(u_h^k)} \tilde u_h^{k+1}.$$
We can use this relation to get a system for the coefficient vector of the function $p_h^{k+1}$.

\begin{lemma}
The coefficient vector of the adjoint state $p_h^{k+1}$ satisfies
$$\underline{p_h^{k+1}} = K^{-1}M \big( K^{-1}(g+M_I \underline{\tilde u_h^{k+1}})-\underline{z_h} \big).$$
\end{lemma}

Note that only the adjoint state $p_h^{k+1}$ is used to update the active and inactive sets, hence kinks and discontinuities will not be accumulated.

As mentioned in \cite{beuchler2012} the operator on the left-hand side of \eqref{equ:Control_lin_sys} is positive definite on $L^2(I(u^k))$, hence the matrix on the left-hand side of \eqref{equ:Control_Matrix_Sys} is positive definite on the span of all basis functions whose support has non-empty intersection with the inactive set $I(u_h^k)$. This makes the equation accessible with a conjugate gradient method.

With these results we can implement our Newton method and solve the subproblem without actually computing $u_h^k$, we only work with adjoint state and the active/inactive sets.

\subsection{Using the Newton-solver in the Bregman iteration}

The fact that we are not computing the control (which is not a FEM function) and work instead with the adjoint state (which is a FEM function) can be extended to the implementation of the Bregman iterative method. Denote $k$ the number of iterations and let $\lambda_h^k \in V_h$ be the computed subgradient. Let $p_h^{k+1} \in V_h$ be the adjoint state computed while solving the subproblem. To start the next iteration all we have to do is to update the subgradient
$$\lambda_h^{k+1} := - \frac{1}{\alpha_{k+1}} p_h^{k+1} + \lambda_h^{k} \in V_h,$$
to start the next iteration. Again note that we do not need to compute the control. The control can be computed (for plotting e.g.) using the optimality condition $u_h^{k+1} = P_\Uad(\lambda_h^{k+1})$ if needed.

\section{Numerical examples}\label{sec:numerical_results}
In the following we present some numerical examples for the problem

\begin{equation}\label{eq:test_problem}
 \begin{split}
    \text{Minimize} &\quad \frac{1}{2}\|y - z\|_Y^2  \\
    \text{such that} & \quad -\Delta y = u + e_\Omega \quad \text{in } \Omega,\\
    &\quad y=0 \quad \text{on } \partial \Omega,\\
    &\quad u_a \leq u \leq u_b \quad \text{a.e. in } \Omega,
\end{split}
\end{equation}

to illustrate our stopping rule. Our implementation is done in FEniCs (\cite{fenics}) and we use Lagrange polynomials of order 1. In one space dimensions we use an equidistant mesh, and in two space dimensions we use a regular triangulation. Here the degrees of freedom of our discretization will be denoted by $DOF$. Let $z_h \in V_h$ be given and $\underline{z_h}$ its coefficient vector. Let $n_h \in V_h$ be such that each component of $\underline{n_h}$ is a random number in the interval $[-1,1]$. For a given $\delta$ we set
$$z^\delta_h := z_h + \delta \frac{n_h}{\|n_h\|},$$
to obtain $\|z_h^\delta - z_h\| = \delta$.

\subsection{Example 1}
One can see, that with the choice of $\Omega = (-1,1)$, $u_a = 0$, $u_b=\frac{1}{10}$ and
\begin{align*}
u^\dagger(x) &= \begin{cases}
\frac{1}{10} & \text{if} \quad x \in [-1, -\frac{1}{2}]\\
0 &\text{if} \quad x \in [-\frac{1}{2}, \frac{1}{4}]\\
(x+1)(x-\frac{1}{4})(x-\frac{3}{4})(x-1) &\text{if} \quad x \in [\frac{1}{4}, \frac{3}{4}]\\
0 &\text{if} \quad x \in [\frac{3}{4}, 1]
\end{cases},\\
y^\dagger(x) &= (-1)\begin{cases}
\frac{7}{3072} + \frac{803 x}{15360} + \frac{x^2}{20} & \text{if} \quad x \in [-1, -\frac{1}{2}]\\
- \frac{157}{15360} + \frac{7x}{3072} &\text{if} \quad x \in [-\frac{1}{2}, \frac{1}{4}]\\
-\frac{581}{49152} + \frac{11x}{480} - \frac{3 x^2}{32} + \frac{x^3}{6} - \frac{13x^4}{192} - \frac{x^5}{20} + \frac{x^6}{30} &\text{if} \quad x \in [\frac{1}{4}, \frac{3}{4}]\\
-\frac{271}{15360} + \frac{271 x}{15360} &\text{if} \quad x \in [\frac{3}{4}, 1]
\end{cases}\\
p^\dagger(x) &= (-1)\begin{cases}
-(x+1)(x+ \frac{1}{2})^3(x- \frac{1}{4})^4 & \text{if} \quad x \in [-1, \frac{1}{4}]\\
0 &\text{if} \quad x \in [\frac{1}{4}, \frac{3}{4}]\\
(x-1)(x-\frac{3}{4})^4 &\text{if} \quad x \in [\frac{3}{4}, 1]
\end{cases},\\
z(x) &= y^\dagger(x) - \Delta p^\dagger(x)\\
e_\Omega(x) &= 0
\end{align*}

the functions $(u^\dagger, y^\dagger, p^\dagger)$ are a solution to \eqref{eq:test_problem}. Furthermore Assumption \ref{ass:ActiveSet} is satisfied with $I = (0,1)$ and $\kappa = \frac{1}{4}$, hence the solution is bang-bang on $A = (-1,0]$ and satisfies a source condition on $I$. We apply algorithm \ref{alg:MinEx} with constant $\alpha_k = 1$ and different noise level $\delta$ and compute $\|u^\dagger - u_k^\delta\|$. Furthermore we compute the stopping rule with $\tau = 5 \cdot 10^3$ and $DOF=10^5$. The results can be found in figure \ref{fig:case12}.

\subsection{Example 2}
With the choice of $\Omega = (-1,1)$, $u_a = -1$, $u_b = 1$ and
\begin{align*}
p^\dagger(x) &= \sin(\pi x)\\
u^\dagger(x) &= - \text{sign}(p^\dagger)\\
y^\dagger(x) &= 1-x^2\\
e_\Omega(x) &= - \Delta y^\dagger(x) - u^\dagger(x)\\
z(x) &= y^\dagger(x) + \Delta p^\dagger(x)
\end{align*}
the functions $(u^\dagger, y^\dagger, p^\dagger)$ are a solution to \eqref{eq:test_problem}. Here the solution satisfies Assumption \ref{ass:ActiveSet} with $A=\Omega$ and $\kappa = 1$. Again we apply algorithm \ref{alg:MinEx} with constant $\alpha_k = 1$ and different noise level $\delta$. The stopping rule $k(\delta)$ is computed with $\tau = 10^6$ and $DOF=10^5$. The results can be found in figure \ref{fig:case12}.

\begin{figure}[htbp]
\makebox[\textwidth]{\input{example1.tikz}}\vspace{0.5cm}
\makebox[\textwidth]{\input{example2.tikz}}
\caption{Regularization error $\|u_k^\delta-u^\dagger\|$ of example 1 and 2, with different noise estimates $\delta$ after 750 Iterations. The markers highlight the stopping points using the a priori stopping rule.}
\label{fig:case12}
\end{figure}

\subsection{Example 3}
With the choice of $\Omega = (0,1)$, $u_a = -1$, $u_b = 1$ and
\begin{align*}
p^\dagger(x) &= x(1-x)(3x-1)^3\\
u^\dagger(x) &= - \text{sign}(p^\dagger(x))\\
y^\dagger(x) &= 1-x^2\\
e_\Omega(x) &= - \Delta y^\dagger(x) - u^\dagger(x)\\
z(x) &= y^\dagger(x) + \Delta p^\dagger(x)
\end{align*}
the functions $(u^\dagger, y^\dagger, p^\dagger)$ are a solution to \eqref{eq:test_problem}. Here the solution satisfies Assumption \ref{ass:ActiveSet} with $A=\Omega$ and $\kappa = \frac{1}{3}$. Again we apply algorithm \ref{alg:MinEx} with constant $\alpha_k = 1$ and different noise level $\delta$. The stopping rule $k(\delta)$ is computed with $\tau = 5 \cdot 10^5$ and $DOF=10^5$. The results can be found in figure \ref{fig:case34}.

\subsection{Example 4}
With the choice of $\Omega = (0,1)^2$, $u_a = -1$, $u_b = 1$ and
\begin{align*}
p^\dagger(x) &= - \frac{1}{8 \pi^2} \sin(2\pi x) \sin(2\pi y)\\
u^\dagger(x) &= - \sign(p^\dagger(x))\\
y^\dagger(x) &= \sin(\pi x) \sin(\pi y)\\
e_\Omega(x) &= 2 \pi^2 \sin(\pi x) \sin(\pi y) - u^\dagger\\
z(x) &= \sin(\pi x) \sin(\pi y) + \sin(2\pi x) \sin(2\pi y)
\end{align*}
the functions $(u^\dagger, y^\dagger, p^\dagger)$ are a solution to \eqref{eq:test_problem}. Here the solution satisfies Assumption \ref{ass:ActiveSet} with $A=\Omega$. Numerical estimates indicate $\kappa = 1$. Again we apply algorithm \ref{alg:MinEx} with constant $\alpha_k = 0.1$ and different noise level $\delta$. The stopping rule $k(\delta)$ is computed with $\tau = 10^7$ and $DOF=10^6$. The results can be found in figure \ref{fig:case34}.

\begin{figure}[htbp]
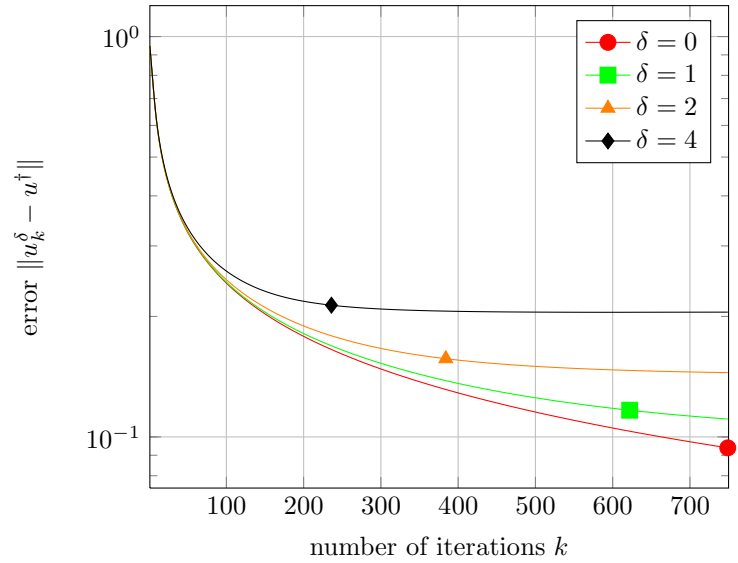

\makebox[\textwidth]{\input{example3.tikz}}\vspace{0.5cm}
\makebox[\textwidth]{\input{example4.tikz}}
\caption{Regularization error $\|u_k^\delta-u^\dagger\|$ of example 3 and 4, with different noise estimates $\delta$ after 750 Iterations. The markers highlight the stopping points using the a priori stopping rule.}
\label{fig:case34}
\end{figure}

Let us remark that such an a priori stopping rule is barely possible in practice, as the constant $\kappa$ appearing in Assumption \ref{ass:ActiveSet} is not known a priori, as it depends on the unknown solution of the unregularized problem and the possible unaccessible noiseless data. Furthermore the choice of $\tau$ is not clear a priori. Nevertheless we can use the a priori rule as an benchmark to compare the convergence order of an a posteriori stopping rule.

In theorem \ref{thm:asymptotic_convergence} we proved asymptotic convergence of our stopping rule independent from $\tau$. This can also be observed numerically. We computed $\|u_{k(\delta)}^\delta - u^\dagger\|$ for different values of $\tau$ and $\delta$ for example 3 with constant $\alpha_k=0.1$. The results can be found in figure \ref{fig:asymptotic}.

\begin{figure}[htbp]
\makebox[\textwidth]{\begin{tikzpicture}
\begin{loglogaxis}[
    title=Asymptotic convergence for Example 3,
    xlabel={error estimate $\delta$},
    ylabel style={align=center}, 
    ylabel=error $\|u_{k(\delta)}^\delta-u^\dagger\|$\\\\,
    xminorticks=true,
    xmin=0.0009765625,
    xmax=1,
    log basis x={2},
    height=8cm,
    grid=major,
    legend entries={$\tau=10^0$,$\tau=10^1$,$\tau=10^2$,$\tau=10^3$,$\tau=10^4$},
    legend pos=north west,
]

%
\addplot[color=red, mark=*, mark options={scale=1.5}] plot coordinates {
(1,0.639332383891995)
(0.5,0.6388888583111267)
(0.25,0.6386690604249581)
(0.125,0.4740306917339204)
(0.0625,0.4372910754618835)
(0.03125,0.3904109755868705)
(0.015625,0.3538590196793837)
(0.0078125,0.3211422421756865)
(0.00390625,0.2895970747875865)
(0.00195312,0.261892378248554)
(0.000976562,0.2370930143784563)
};

\addplot[color=green, mark=square*, mark options={scale=1.5}] plot coordinates {
(1,0.639332383891995)
(0.5,0.5050740474995314)
(0.25,0.4534989561546362)
(0.125,0.4053288520150629)
(0.0625,0.3692483033738199)
(0.03125,0.3326952165547493)
(0.015625,0.2999376814262183)
(0.0078125,0.2712460083446118)
(0.00390625,0.2452542259625069)
(0.00195312,0.22178904694237)
(0.000976562,0.2006931911215759)

};

\addplot[color=orange, mark=triangle*, mark options={scale=1.5}] plot coordinates {

(1,0.4758303237784214)
(0.5,0.4258065500863034)
(0.25,0.3795281049604838)
(0.125,0.3452801481247477)
(0.0625,0.3109000007738519)
(0.03125,0.2808832578047406)
(0.015625,0.2538964079636459)
(0.0078125,0.2296179077490553)
(0.00390625,0.2076637812531929)
(0.00195312,0.1878833807890732)
(0.000976562,0.1700494606683246)

};

\addplot[color=black, mark=diamond*, mark options={scale=1.5}] plot coordinates {
(1,0.3942044027697085)
(0.5,0.3561976379610592)
(0.25,0.3225802462935994)
(0.125,0.291344057138852)
(0.0625,0.2633130159053495)
(0.03125,0.2378945932606141)
(0.015625,0.2150559492158381)
(0.0078125,0.194490719248928)
(0.00390625,0.1759812336769707)
(0.00195312,0.1592634405789415)
(0.000976562,0.1441588448966071)

};

\addplot[color=black, mark=square, mark options={scale=1.5}] plot coordinates {
(1,0.3404111210638429)
(0.5,0.3056236413081528)
(0.25,0.2745467920789085)
(0.125,0.2474523367690911)
(0.0625,0.223220181166162)
(0.03125,0.2016890121446885)
(0.015625,0.1823810380160502)
(0.0078125,0.1649608128967995)
(0.00390625,0.1492628430668249)
(0.00195312,0.1350831027102612)
(0.000976562,0.1222726356628757)

};

\end{loglogaxis}
\end{tikzpicture}}
\caption{Asymptotic convergence of the stopping rule for example 3. Here we use $DOF = 10^5$ and $\alpha_k=0.1$.}
\label{fig:asymptotic}
\end{figure}
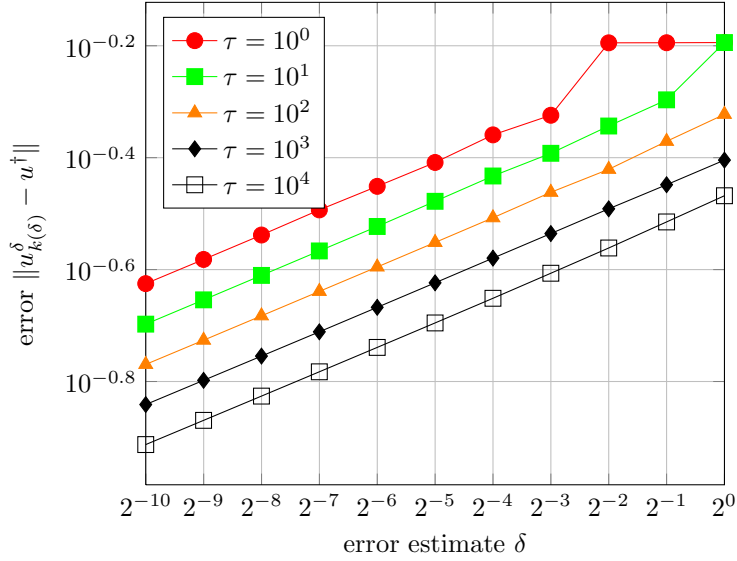

\begin{remark}
Let us remark that it is an open question to construct an a posteriori stopping rule in our case. It is not clear how to extend the a posteriori estimates presented in \cite{wachsmuth2011} into our iterative method. Furthermore we cannot apply more general a posteriori stopping rules, as presented in \cite{schock2005} as they rely on estimates of $\|u_k - u_k^\delta\|$, which are not available in our case.
\end{remark}

\newpage
\bibliographystyle{plain}
\bibliography{literatur}

\end{document}